\documentclass[12pt]{amsart}

\usepackage{amsmath,amssymb,amscd,amsfonts}
\usepackage{anysize}
\marginsize{28mm}{24mm}{19mm}{30mm}

\usepackage{amsmath,amssymb,amscd,endnotes,graphicx}
\usepackage{mathrsfs}  
\usepackage[mathscr]{euscript}

\usepackage{graphics}
\newtheorem*{theorem*}{Theorem}
\newtheorem*{lemma*}{Lemma}
\newtheorem{cor}{Corollary}
\newtheorem{theorem}{Theorem}
\newtheorem{lemma}{Lemma}

\theoremstyle{definition}

\theoremstyle{remark}
\newtheorem*{defstar}{Definition}
\numberwithin{equation}{section}

\newcommand{\Z}{\mathbb{Z}}

\newcommand{\ccirc}{\kern0.5ex\vcenter{\hbox{$\scriptstyle\circ$}}\kern0.5ex}

\def\det{\operatorname{det}}

\newcommand{\s}{\sigma}

\def\Z{\mathbb Z}

\def\Cc{\mathcal C}
\def\R{\mathbb R}

\def\D{\Delta}
\def\s{\sigma}

\def\z{\zeta}
\def\d{\delta}

\def\e{\epsilon}

\newcommand{\defeq}{\mathrel{\mathop:}=}

\begin{document}

\title[On the analytic theory of  isotropic ternary quadratic forms II]{On the analytic theory of isotropic ternary quadratic forms II}
\author{W. Duke}
\address{UCLA Mathematics Department,
Box 951555, Los Angeles, CA 90095-1555} \email{wdduke@ucla.edu}
 
\begin{abstract}
In  the first part of this work \cite{Du},  a quantitative supplement to the Hasse principle was given for the count of  the number of automorphic orbits of  primitive  zeros of a genus of ternary  quadratic forms. 
This sequel contains, for certain special forms,  an independent and elementary proof  of this result.  When combined with other results of \cite{Du}, this proof also leads to a refinement of an asymptotic result of \cite{Du} and  some corollaries for these special forms.
\end{abstract}
\maketitle

%

\section{Introduction}

One of the main results of  \cite{Du} is that the number of ($\Z$-automorphic) orbits of primitive integral zeros of a genus of  isotropic ternary quadratic forms equals the product, taken over primes $p$,  of the number 
of orbits of these zeros under $\Z_p$-automorphs.
In more detail, let $S$ be a symmetric $3\times 3$ matrix with integral entries and with $\det{S}=D>0$.
Associated to $S$ is the nonsingular ternary quadratic form
\begin{align}\label{dio1}
S(x)\defeq xSx^t=x\left( \begin{smallmatrix}
a&d&e\\
d&b&f\\e&f&c
\end{smallmatrix}\right)x^t \nonumber=a x_1^2+b x_2^2+c x_3^2+2 d x_1x_2+2e x_2x_3+2fx_1x_3,
\end{align}
where $x=(x_1,x_2,x_3)$ and $a,b,c,d,e,f\in \Z$.  Two forms $S,S'$ are in the same class  if there is an $A \in \mathrm{GL}_3(\Z)$ with 
\begin{equation}\label{cl}A^tSA=S'.\end{equation} They are in the same genus if for each prime $p$ (and $p=\infty$) there is an $A \in \mathrm{GL}_3(\Z_p)$ such that (\ref{cl}) holds. 
The genus $G$ of $S$ consists of finitely many, say h, classes \cite[Cor 1, p.139]{Cas}.
The form $S$ is isotropic if $S(x)=0$ for some primitive $x\in \Z^3.$
Let $\Cc(S)$ be the set of all primitive $x\in \Z^3$ with  $S(x)=0$
  and $O$ the group of integral automorphs of $S$, which is  given by
 \begin{equation}\label{auto}O=O(S)=\{A \in \mathrm{GL}_3(\Z);\;A^tSA=S\}.\end{equation}
Say   $x,x'\in \Cc(S)$  are in the same $\Z$-orbit if there is an $A \in O$ such that 
$
xA^t=x'.
$
Let $c(S)$ denote the number of $\Z$-orbits.  For a prime $p$ or $p=\infty$ 
denote by  $c_p(S)$  the number of $\Z_p$-orbits, which is defined similarly.  Here $c_p(S)\neq1$ for at most finitely many $p$. The following result, which was given as Theorem 2 of \cite{Du}, is a quantitative supplement to the Hasse principle for ternary quadratic forms.
 \begin{theorem*}\label{t1}
 For  $S(x)$ a nonsingular integral ternary quadratic form associated to a matrix $S$  as above we have 
\begin{equation}\label{cn2}
\sum_{S'\in G}c(S')=\prod_{p} c_p(S),
\end{equation}
where $G$ is the genus containing $S$. 
\end{theorem*}

Implicit here is the statement that each $c(S')$ is finite.
The proof of this theorem given in \cite{Du}  combines algebraic and analytic methods.
In this companion paper to \cite{Du},  I will restrict attention to   $S$ that satisfy certain conditions that are special but not overly restrictive. 
The goal is twofold. First I   give a proof of this result  for these $S$ that is  independent of  \cite{Du} and is more elementary and direct.
The idea is to use the extra conditions on $S$ to  extend the reduction process used in \S 5 of \cite{Du} and explicitly classify the orbits. 
 Second, 
for such $S$,  I will refine an asymptotic formula of  \cite{Du}  to show that the solutions are uniformly distributed across the orbits, which is  a property that does not hold for a  general $S$.   Also, I  deduce from this a simple explicit formula for the number of orbits. 
These new results follow easily once we have the elementary proof of  the theorem  and can use it in combination with statements from  \cite{Du}.
\section{New results for special forms}

Suppose that $S$ is primitive and isotropic  with $D=\det S>0$.
Let $-\Omega>0$ be the GCD of the $2\times 2$ minors of $S$.  It is classical, and will be recalled below, that for some $\D\in \Z^+$
\[
D=\det{S}=\Omega^2\D.
\]
  Note that the sign of $\Omega$ is chosen to be negative so that is conforms to the classical convention made in \cite{Smi} and \cite{Dic}.  
  The values $\Omega$ and $\D$ are genus invariants. 
  For convenience, in this paper I will use the following terminology.
  \begin{defstar}
  Say that $S$ is {\it special} if it is primitive, isotropic with $D>0$ and if
    \[N\defeq \gcd(\Omega,\D), \;\;\tfrac{\Omega}{N}\;\;\mathrm{and} \;\;\tfrac{\D}{N}\]  are  square-free.  \end{defstar}
    In particular, $S$ is special  if   $\Omega$ and $\D$  are square-free and this certainly holds if $D$ is square-free. 
   In the next section we will give an elementary and independent proof of the Theorem of the Introduction for special $S(x).$

    \begin{theorem}\label{t1}
 For  special $S(x)$  with odd $D$ we have
\begin{equation}\label{cn2}
\sum_{S'\in G}c(S')=\prod_{p} c_p(S),
\end{equation}
where $G$ is the genus containing $S$. 
\end{theorem}

   In fact, the proof yields that for these forms 
   \begin{equation}\label{inf}
 \sum_{S'\in G}c(S')= \prod_{p|N}\tfrac{p-1}{2}.
   \end{equation}


The proof of the theorem  given in \cite{Du} makes use of an asymptotic formula that counts primitive points in an orbit. For special $S$ the constant in this asymptotic depends only on the genus $G$. 
  Let \[S^*=(\det{S})S^{-1}\] be the adjugate of $S$. Choose $y\in \R^3$ such that $S^*(y)=4D.$ It is easy to show that are at most finitely many $x\in \Cc(S)$ with
 $
 0<xy^t\leq T
 $
 for $T>0.$     
 For $x\in \Cc(S)$ its orbit is defined by
 \[\Cc(S,x)=\{x'\in \Cc(S); x'=xA^t\;\;\text{for some}\;\;A\in O(S)\},\]
 where $O(S)$ was defined in (\ref{auto}). 

\begin{theorem}\label{thm3}Let   $S$ be special and   fix $y\in \R^3$ with $S^*(y)=4D$.  Then for any  $x\in \Cc(S)$
we have
  \begin{equation}\label{mt}
  \#\{x'\in \Cc(S,x);\, 0<x'y^t \leq T\}\sim \kappa
 \, T,
  \end{equation}
  as $T\rightarrow \infty$, where $\kappa>0$ depends only on the genus $G$ of $S$.
    \end{theorem}
   
   The value of $\kappa$ can be given explicitly in terms of  $h$, the class number of $G$,  and the primes dividing $2D$.  In case $\Omega$ and $\D$ are square-free and odd, it is given by
   \begin{equation}\label{kappa}
  \kappa=\tfrac{3h}{2\pi\sqrt{D}} \prod_{p|D} \tfrac{2p}{p+1}\prod_{p|N}\tfrac{2}{p-1}.
   \end{equation}

Since for special $S$ the constant  $\kappa$ in (\ref{mt}) does not depend on $x$, the following consequence is immediate.
\begin{cor}\label{cor2}
For special $S$ the distribution of the primitive solutions is uniform across the orbits. Explicitly,
\[
\rho(x)\defeq\lim_{T\rightarrow \infty} \frac{\#\{x'\in \Cc(S,x);\, 0<x'y^t \leq T\}}{\#\{x'\in \Cc(S);\, 0<x'y^t \leq T\}}=\frac{1}{c(S)}.
\]
\end{cor}

This corollary is not  true for general $S$. 

\subsubsection*{Example i)}
 It  will be shown in \S \ref{se4}  that the orbits of
\begin{equation}\label{S}
S(x)=-3^3x_1^2+x_2^2-x_3^2,
\end{equation}
which is not special, 
   are represented by  $x,x',x''\in \Cc(S)$    
  where \begin{equation}\label{rho}
\rho(x)=\rho(x')=\tfrac{1}{5}\;\;\mathrm{and}\;\;\rho(x'')=\tfrac{3}{5}.
\end{equation}

\smallskip
 Theorem \ref{thm3}, together  with (\ref{inf}) and the fact that the constant in the asymptotic of the full count given in Theorem 3 of \cite{Du} is a genus invariant,  imply  the following simple formula for the number of orbits of $S$. 
\begin{cor}\label{thm2}  Let $S$  be special  with $D=\det S$ odd.
   Then
\begin{equation*}\label{cn3}
c(S)=h^{-1}\prod_{p|N}\tfrac{p-1}{2},
\end{equation*}
where $h$ is the class number of the genus of $S$ and $N=\gcd(\Omega,\D)$.
\end{cor}

This result leads to  different justifications of  Examples i) and ii) given  in \cite{Du} after Theorem \ref{t1} and makes it easy to give more. 

\subsubsection*{Example ii)}
Consider the Legendre equation 
\begin{equation*}\label{exx}
S(x)=q^2x_1^2-qx_2^2-x_3^2=0,
\end{equation*}
where $q$  is a product of $\nu$ distinct primes,  each  $\equiv 1\,(\mathrm{mod}\,8)$ and each a quadratic residue of every other. It follows from \cite[p.188]{Mey} that $h=2^{\nu}$. This $S$ is special with $N=q$.  Thus by Corollary \ref{thm2}
we have that
\begin{equation*}\label{mm1}
c(S')=\prod_{p|q} \tfrac{p-1}{4}
\end{equation*}
for any $S'$ in the genus of $S$.

\section{Elementary proof of Theorem \ref{t1} }
  The proof relies on ideas from  papers of Eisenstein \cite{Eis}, Smith \cite{Smi2} and especially A. Meyer \cite{Mey}. Part of it, specifically i) of Lemma \ref{l1},  will also be used in the proof of Theorem \ref{thm3}.  In addition, (\ref{inf}) comes out automatically. 
The proof can be simplified if we  assume  that $\Omega$ and $\D$ are square-free, but for the statement of Theorem \ref{thm3} and its corollaries it is desirable to make as few assumptions as is possible.

%

Recall that $D=\det\,S>0$. The  primitive adjugate of $S$  is \[S^{\dag}\defeq \Omega^{-1}S^*,\]which has discriminant $\Omega^{-3}D^2$. 
 The GCD of the $2\times 2$ minors of $S^*$ is $D$ so
 the GCD of  the entries of the adjugate of $S^\dag$  is the integer $\D=D\Omega^{-2}$. Therefore  $D=\Omega^2\D$.
 
We are assuming  that \[N=\gcd(\Omega,\D),\;\tfrac{\Omega}{N},\;\tfrac{\D}{N}\;\;\text{are square-free}.\]
Define $N_5=\gcd(N,\tfrac{\Omega}{N})$,   $N_4=\gcd(N,\tfrac{\D}{N})$
and positive integers $N_1,N_2,N_3$ through
\begin{equation}\label{NN}
N=N_3N_4N_5, \;\;\;\; \Omega=-NN_5N_2,\;\;\;\;\D=NN_4N_1.
\end{equation}
Then $N_1,N_2,N_3,N_4,N_5$ are square-free and relatively prime in pairs.
Also,
\begin{equation}\label{N2}
D=N_1N_2^2N_3^3N_4^4N_5^5,
\end{equation}
which explains the choice of subscripts of the $N_j's.$
Note that from (\ref{NN}) we have
\begin{equation}\label{NNN}\Omega=-N_5^2N_4N_3N_2\;\;\;\mathrm{and}\;\;\D=N_4^2N_5N_3N_1.\end{equation}
In case $\Omega$ and $\D$ are square-free we have that $N_4=N_5=1$ and $N=N_3$ so
\[
\Omega=-NN_2\;\;\;\;\mathrm{and}\;\;\; \D=NN_1.
\]

\begin{lemma} \label{l1}
Suppose that $S$ is special with $D=\Omega^2\D$. 

i) To each orbit of primitive solutions $x$ of $S(x)=0$ corresponds $\ell\in \Z$ such that  $S$  is properly equivalent to
\begin{equation}\label{S3}
S_3=\left(\begin{smallmatrix}
   0& 0&N_3N_5N_4^2N_2 \\
     0&-N_3N_5^3N_1&0\\
 N_3N_5N_4^2N_2&0&N_4N_2\ell
\end{smallmatrix}\right).
\end{equation}
Here  $S_3$ is  uniquely determined  up to the value of $\ell$.

ii) If $D$ is odd  we may assume that $0\leq\ell< N$ with $\gcd(\ell,N)=1$ and then $\ell$ is uniquely determined.


iii) The primitive adjugate of $S_3$ is given by  \begin{equation}\label{Sp2}
S_3^\dag=\left(\begin{smallmatrix}
   -\ell N_5N_1& 0&N_3N_5^2N_4N_1 \\
     0&-N_3N_4^3N_2& 0\\
N_3N_5^2N_4N_1& 0& 0
\end{smallmatrix}\right).
\end{equation}

\end{lemma}

\begin{proof}

Choose any primitive $x\in \Z^3$ with $S(x)=0$.
  Completing $x^t$  to $M_1\in \mathrm{SL}(3,\Z)$ 
we have
\begin{equation*}\label{s1}
S[M_1]=\left(\begin{smallmatrix}
   0& s_1&s_2\\
     s_1&*& *\\
 s_2& *& *
\end{smallmatrix}\right).
\end{equation*}
Suppose $a=\gcd(s_1,s_2).$ Choose $u,v\in \Z$ with $us_1+vs_2=a.$
Define
\[
M_2=\left(\begin{smallmatrix}
   1& 0&0\\
     0&\tfrac{s_2}{a}& u\\
 0& -\tfrac{s_1}{a}& v
\end{smallmatrix}\right)\in  \mathrm{SL}(3,\Z)
\]
so that $M_1M_2$ still has $x^t$ as its first column and 
 \begin{equation}\label{ss1}S_1=S[M_1M_2]=\left(\begin{smallmatrix}
   0& 0&a \\
     0&-b& c\\
 a& c& d
\end{smallmatrix}\right).\end{equation}
This shows that we may find an equivalent $S_1$ of the form (\ref{ss1}) whose associated  orbit 
of primitive zeros obviously contains  $(1,0,0)$. 
It easily follows that $a,b>0$ are uniquely determined,  and that $c$ is determined modulo $\gcd(a,b)$, {\it  by the orbit of the solution. }

We will show that our conditions on $S$ imply that $a$ and $b$ are determined by $\Omega$ and $\Delta$ and can be written in terms of the $N's.$
Explicitly, we will show that
\begin{equation}\label{ad7} a=N_3N_5N_4^2N_2\;\;\mathrm{and}\;\;b=N_3N_5^3N_1.\end{equation}
Recall that  $-\Omega$ is the gcd of the entries of $S^*,$ hence of 
\begin{equation}\label{ad}
S_1^*=\left(\begin{smallmatrix}
-bd-c^2 & ac& ab\\
ac& -a^2 & 0\\
 ab & 0 & 0\\
\end{smallmatrix}\right).
\end{equation}
In particular, $\Omega|a^2$. It follows from (\ref{NNN}) that the square part of $\Omega$ is $N_5^2$, so
\begin{equation}\label{ad3} a=N_5N_4N_3N_2a'\end{equation} for some $a'\in \Z.$ From (\ref{N2})
\begin{equation}\label{ad2}
a'^2b=N_1N_3N_4^2N_5^3.
\end{equation}
Hence $(N_3N_5N_1)|b.$ Next we show that $N_5^3|b.$ Otherwise there would be a prime $p|N_5$ and $p|b$ such that $p^2\nmid b$. But from the divisibility of each entry of $S_1^*$ in (\ref{ad}) by $\Omega$
in particular that
$
p^2|(-bd-c^2),
$
we would have that $p|d$ and $p|c$.  Also, $N_5|a$ so $p|a$, contradicting the primitivity of $S$.
Therefore by (\ref{ad2}) we have 
\begin{equation}\label{ad5}
b=N_3N_5^3N_1b'^2
\end{equation}
for some $b'\in \Z$ and also  $a'b'=N_4$. We claim that $b'=1$ and so $a'=N_4.$
For this, note that $b'$ is square-free and 
\[
N_4|(-bd-c^2)
\]
so $b'|c$. 
Now $b'|N_4$ and so $b'^2\nmid \Omega$ and from (\ref{ad3}) we have that $b'|a.$
It follows that $b'$ divides all coefficients of $S_1^\dag$,  a contradiction. 
Thus (\ref{ad7}) follows from (\ref{ad3}) and (\ref{ad5}).

Furthermore, from $\Omega |(-bd-c^2), (N_3N_5)|\Omega$ and $(N_3N_5)|b$ we have for $c$  given in   (\ref{ss1}):  
\begin{equation}\label{ad8}
c=N_3N_5c'.
\end{equation}
for some $c'\in \Z.$
Next we want to transform $S_1$ in (\ref{ss1})  so that $c$ becomes zero.
For $t_1\in \Z$ and
\[
T_1=\left(\begin{smallmatrix}
   1& 0&0\\
     0&1 & t_1\\
 0& 0& 1
\end{smallmatrix}\right)
\;\;\text{we have}\;\;
S_2=S_1[T_1]=\left(\begin{smallmatrix}
   0& 0&a \\
     0&-b& c-bt_1\\
 a& c-bt_1&d'
\end{smallmatrix}\right),\]
where \begin{equation*}
d'=-bt_1^2+2ct_1+d.
\end{equation*}
From (\ref{ad7}) and (\ref{ad8}) we obtain
\[
c-bt_1=N_3N_5(c'-N_5^2N_1t_1).
\]
Choose $t_1$ so that $N_4^2N_2| (c'-N_5^2N_1t_1)$. Then for some $a''\in \Z$ we have
$
c-bt_1=a'' a.
$
Now for
\[
T_2=\left(\begin{smallmatrix}
   1& -a''&0\\
     0&1 & 0\\
 0& 0& 1
\end{smallmatrix}\right)\;\;\text{it holds that}\;\;S_3=S_2[T_2]=\left(\begin{smallmatrix}
   0& 0&a \\
     0&-b&0\\
 a&0&d'
\end{smallmatrix}\right).
\]
A calculation shows that \[S_3^*=\left(\begin{smallmatrix}
   -bd'&0&ab \\
     0&-a^2&0\\
 ab&0&0
\end{smallmatrix}\right)\]
so $\Omega|(bd')$ or $(N_5^2N_4N_3N_2)|(N_3N_5^3N_1d')$, hence  $d'=N_4N_2\ell$ for some $\ell\in \Z$. Expressed in terms of the $N_j's$ we have (\ref{S3}). 
This gives the first statement of the lemma.

\smallskip
Suppose now that $D$ is odd. 
Now for
\[
T_3=\left(\begin{smallmatrix}
   1&N_5^2N_1&\frac{1}{2}(1+N_5^2N_4^2N_2N_1)\\
     0&1 &N_4^2N_2\\
 0& 0& 1
\end{smallmatrix}\right),
\]
which is integral since $D$ is odd, and $k\in \Z$ we have
\[
S_3[T_3^k]=\left(\begin{smallmatrix}
   0& 0&N_3N_5N_4^2N_2 \\
     0&-N_3N_5^3N_1&0\\
 N_3N_5N_4^2N_2&0&N_4N_2(\ell+kN)
\end{smallmatrix}\right).\]
Thus we can find a $k$  to reduce $S_3$ to the form  (\ref{S3}) where $0\leq \ell<N$. 
Now $S_3^\dag$ in (\ref{Sp2}) is easily computed. Using the primitivity of both $S_3$ and $S_3^\dag$ we see that 
 $\gcd(\ell,N)=1.$
The value of $\ell$ is uniquely determined by $S_3$ in (\ref{S3}) since the residue class of $\ell$ modulo $N$ is preserved under all transformations of $S_3$ that fix all but its $(3,3)$-entry. 

\end{proof}


In order to prove (\ref{cn2}),  it is convenient to relate the definition of genus for ternary quadratic forms given above with that introduced by Eisenstein  and further developed by Smith. Their definition is in terms of characters and was modelled on that of Gauss  for binary quadratic forms. For ternary forms we must consider simultaneously the form $S$ and its primitive adjugate $S^\dag.$ The definition  is based on the identities
\begin{align*}
S(x)S(y)-\tfrac{1}{4}\big(x_1 \partial_{s_1}& S(y)+x_2 \partial_{y_2}S(y) +x_3 \partial_{y_3}S(y)\big)^2\\&=\Omega S^\dag(x_2y_3-x_3y_2,x_3y_2-x_1y_3,x_1y_2-x_2y_1)\\
S^\dag(x)S^\dag(y)-\tfrac{1}{4}\big(x_1 \partial_{y_1}& S^\dag(y)+x_2 \partial_{y_2}S^\dag(y) +x_3 \partial_{y_3}S^\dag(y)\big)^2\\&=\D S(x_2y_3-x_3y_2,x_3y_2-x_1y_3,x_1y_2-x_2y_1),
\end{align*}
where $x=(x_1,x_2,x_3)$ and $y=(y_1,y_2,y_3).$
For odd $D$, let $p,q$ be primes with $p|\Omega$ and $q|\D$.   Using these identities we see that for  $m$ with $p\nmid m$ represented by $S$,  the value $\big(\frac{m}{p}\big)$ is independent of the choice of $m$,  as is 
$\big(\frac{n}{q}\big)$ when $n$ is represented by $S^\dag$ and $q\nmid n$.
In the Eisenstein/Smith definition, the genus of $S$ consists of all $S'$ with invariants $\Omega,\D$ and these character values for each $p,q$ and where, when $p=q$, we have two character values. 

In \cite{Smi2} Smith  showed that, by this definition, $S$ and $S'$ are in the same genus if and only if there is a rational transformation $A$, whose entries have denominators prime to $2\Omega \D$, 
with $S'=S[A].$ 
As a consequence, the Eisenstein/Smith definition of genus coincides with the usual one. For a proof see e.g. Theorem 50 on p.78. of \cite{Wat} (see also the note on \S 5 on p. 138).

Smith  \cite{Smi} also gave explicit conditions for a ternary form to represent zero nontrivially,  which generalize those in  Legendre's theorem.   
Under our assumptions, $S$ represents zero nontrivially  if and only if
\[
\big(\tfrac{S(x)}{p}\big)=\big(\tfrac{N_3N_5N_1}{p}\big),\;\big(\tfrac{-S^\dag(y)}{q}\big)=\big(\tfrac{N_3N_4N_2}{q}\big)  \;\;\;\mathrm{and}\;\;\;\big(\tfrac{S(z)}{r}\big)\big(\tfrac{-S^\dag(z)}{r}\big)=\big(\tfrac{-N_5N_4N_2N_1}{r} \big)
\]
for all primes $p,q,r$ with $p|N_4N_2$ and $x$ with $p\nmid S(x)$,  $q|N_5N_1$ and $y$ with $q\nmid S^\dag(y)$ and $r|N_3$ and $z$  with $r\nmid S(z)S^\dag(z).$

It follows that the genus characters of  $S$ are determined by the values of
\[
\big(\tfrac{S(x)}{p}\big)\;\;\;\mathrm{and}\;\;\;\big(\tfrac{-S^\dag(y)}{q}\big),\]
for primes $p|N_5N_3$, $q|N_4N_3$ and any $x $ with $\gcd\big(S(x),N_3N_5\big)=1$ and $y$ for which $\gcd\big(S^\dag(y),N_3N_4\big)=1.$
Referring to Lemma \ref{l1} and the $(3,3)$-entry of $S_3$, together with  the $(1,1)$-entry of $S_3^\dag$, we see that the genus of $S$ is completely determined by the values of 
\[
\big(\tfrac{N_2N_4\ell}{p}\big)\;\;\;\mathrm{and}\;\;\big(\tfrac{N_1N_5\ell}{q}\big)\]
or, equivalently,  by the values of 
$
\big(\tfrac{\ell}{p}\big)$, for $p| N.$

When another $S'$ with corresponding $\ell'\neq \ell$  is properly equivalent to $S$,
  then $\ell$ and $\ell'$  belong to two different orbits of solutions  of the same form.  All forms in the genus of $S$ will be represented.
 The total number of orbits of solutions from forms  of the genus of $S$  is thus 
\[
2^{-\nu(N)} \phi(N)=\prod_{p|N} \tfrac{p-1}{2},
\]
where $\nu$ and  $\phi$ are the usual arithmetic functions.
To finish the proof of Theorem \ref{t1},    repeat the above reduction argument over $\Z_p$ for each $p|N$
to show that \[c_p(S)=\tfrac{p-1}{2}.\]
\qed

\section{Proof of Theorem \ref{thm3} }\label{se4}

 Theorem \ref{thm3}    is an easy consequence of results of \cite{Du}  when combined with the argument above. For special $S$ the values $a,b,c$ of (\ref{ss1}) are completely determined by the genus $G$, as follows from i) of  Lemma \ref{l1}.  Then by Lemmas 3 and 4 of \cite{Du}, we see that  the value $\d_p(S,x)$ is also determined by $G$. It is clear that the value $\s_p$ given in  (3.1) of \cite{Du} depends only on $G$.
Thus Theorem \ref{thm3} follows from 
 Theorem 4  of \cite{Du} and its Corollary 2.
 \qed

The formula (\ref{kappa}) can be derived by using  computations of $\d_p$ from    \cite{Pal}\footnote{ In \cite{Pal} note the corrections: $(n_s-1)$ in (23)  should be $(n_s+1)$ and in the formula for $z_i$ below (47), the number 8 should be 3 and 4 should be 2. These corrections were given in \cite{Wat2}. } and of $\d_p(S,x)$ from \S 3 of \cite{Du}.

\subsubsection*{Justification of Example i)}

It is known (see e.g. \cite[Thm 47]{Dic}) that $h=1$ for $S$ from (\ref{S}). 
 It can be checked that inequivalent  orbits are represented by 
\[
S_1=\left(\begin{smallmatrix}
   0& 0&3\\
     0&-3 & 1\\
 3& 1& 0
\end{smallmatrix}\right),\;\; S_1'=\left(\begin{smallmatrix}
   0& 0&3\\
     0&-3 & 2\\
 3& 2& 0
\end{smallmatrix}\right)\;\;\mathrm {and}\;\;\;S_1''=\left(\begin{smallmatrix}
   0& 0&1\\
     0&-3^3 & 1\\
 1& 1& 1
\end{smallmatrix}\right).
\]
From  Lemmas 3 and 4 of \cite{Du}  we have  that $\prod_p\d_p(S_1,x_0)=\prod_p\d_p(S_1',x_0)=2\cdot3^4$ while $\prod_p\d_p(S_1'',x_0)=2\cdot3^3$, where $x_0=(1,0,0).$
Using   \cite{Pal} and \cite{Jo} in the verification that $\d_2=2$,  we have that
\[
\prod_p\d_p =\tfrac{8}{3\z(2)}\tfrac{3^4}{4}.
\]
Also, it can be shown easily that
\[
\prod_p\s_p=\tfrac{8}{3\z(2)}\tfrac{5}{4}.
\]
Now an application of Theorems 2, 3 and 4 of \cite{Du}  verifies (\ref{rho}) and also that all orbits have been represented.

\bibliographystyle{amsplain}
\bibliographystyle{amsplain}

\end{document}